\documentclass[11pt,a4paper]{article}
\usepackage[utf8]{inputenc}
\usepackage[english]{babel}
\usepackage[left=1.5cm,right=1.5cm,top=2cm,bottom=2cm,foot=1cm]{geometry}
\usepackage{amsmath,amssymb,latexsym, amsthm, mathrsfs}
\usepackage{hyperref}

\newcommand{\vol}{\mathrm{vol}}
\newcommand{\tr}{\mathrm{tr}}

\newcommand{\inj}{\mathrm{inj}}
\newcommand{\ud}{\mathrm{d}}
\newcommand{\id}{\,\mathrm{d}}

\newcommand{\CT}{\mathscr{C}}

\renewcommand{\exp}{\mathop{\rm exp}\nolimits}
\providecommand{\Exp}{\mathop{\rm exp}\nolimits}
\providecommand{\vol}{\mathop{\rm vol}\nolimits}

\providecommand{\tr}{\mathop{\rm tr}\nolimits}

\newtheorem{theorem}{Theorem}
\newtheorem{lemma}{Lemma}
\newtheorem{corollary}{Corollary}

\theoremstyle{definition}
\newtheorem*{rem}{Remark}
\newtheorem*{definition}{Definition}
\title{Harmonic Manifolds and Tubes}
\begin{document}
\author{Bal\'azs Csik\'os\thanks{The first author was supported by the National Research Development and Innovation Office (NKFIH) Grant No. ERC\_HU\_15 118286 and Grant No. OTKA K112703.}, E\"otv\"os Lor\'and University, Budapest, \\ 
    M\'arton Horv\'ath\thanks{The second author was supported by the National Research Development and Innovation Office (NKFIH) Grant No. OTKA K112703.  \newline\textit{2010
    		Mathematics Subject Classification:} Primary 53C25, Secondary
    	53B20.  \newline\textit{Keywords and phrases:} harmonic manifold, D'Atri space, Weyl's tube formula,  Steiner's formula, total mean curvature, total scalar curvature}, Budapest University of Technology and Economics, Budapest} \date{}
\maketitle
\begin{abstract}
The authors showed \cite{tube1} that in a connected locally harmonic manifold, the volume of a tube of small radius about a regularly parameterized simple arc depends  only on the length of the arc and the radius. In this paper, we show that this property characterizes harmonic manifolds even if it is assumed only for tubes about geodesic segments. As a consequence, we obtain similar characterizations of harmonic manifolds in terms  of the total mean curvature and the total scalar curvature of tubular hypersurfaces about curves. We find simple formulae expressing the volume, total mean curvature, and total scalar curvature of tubular hypersurfaces about a curve in a harmonic manifold as a function of the volume density function.
\end{abstract}

\maketitle
\section{Introduction}
Locally harmonic manifolds were introduced by E.~T.~Copson and H.~S.~Ruse \cite{Copson_Ruse} as Riemannian manifolds admitting a non-constant harmonic function in a punctured neighborhood of any point $p$ which depends only on the distance of the variable point from $p$. They showed that locally harmonic manifolds have constant curvature in dimensions $2$ and $3$. In 1944 A.~Lichnerowicz \cite{Lichnerowicz} conjectured that locally harmonic manifolds of dimension $4$ are necessarily locally symmetric spaces and posed the question whether this holds in higher dimensions as well. Locally harmonic manifolds were classified within the family of symmetric spaces by  A.~J.~Ledger \cite{Ledger}, who showed that a locally symmetric space is locally harmonic if and only if it is flat or has rank one.  The Lichnerowicz conjecture was proved by  A.~G.~Walker \cite{Walker} in dimension $4$, and  by Y.~Nikolayevsky \cite{Nikolayevsky} in dimension $5$.  
Z. I. Szab\'o \cite{Szabo} proved the Lichnerowicz conjecture for manifolds having compact universal covering space. By a result of A.-C. Allamigeon \cite{Allamigeon},  a complete simply connected harmonic manifold is either compact or it has no conjugate points along the geodesic curves. G.~Knieper \cite{Knieper} proved that if a compact harmonic manifold has no conjugate points, then it is either flat or a rank one symmetric space, completing the proof of the Lichnerowicz conjecture for all compact harmonic manifolds. As for the non-compact case, the answer to the question of Lichnerowicz is negative in infinitely many dimensions starting at $7$.  E.~Damek and F.~Ricci \cite{Damek_Ricci} noticed that certain solvable extensions of some Heisenberg-type Lie groups  become globally harmonic manifolds if we choose a suitable left invariant Riemannian metric on them, but they happen to be symmetric only if the used Heisenberg-type group has a center of dimension $1$, $3$ or $7$. We refer to the book \cite{DamekRicci} by J.~Berndt, F.~Tricerri, and L.~Vanhecke for more details on Damek--Ricci spaces. At present harmonic symmetric spaces and Damek--Ricci spaces are the only known examples of harmonic manifolds. In 2006 J.~Heber \cite{Heber} showed that a simply connected homogeneous harmonic manifold is either flat, or a rank one symmetric space, or a Damek--Ricci space. The existence of non-homogeneous harmonic manifolds is still an open problem.

There are several notions of harmonicity: infinitesimal, local, global and strong harmonicity, each of which implies the preceding one. As it is remarked by Z.~I.~Szab\'o \cite{Szabo}, infinitesimal and local harmonicity are equivalent properties and they are also equivalent to global harmonicity if the manifold is complete. D.~Michel \cite{Michel} showed that a complete simply connected Riemannian manifold is globally harmonic if and only if it is strongly harmonic. In the rest of the paper, we shall use the term \emph{harmonic manifold} as a short name for locally harmonic manifolds. All manifolds are assumed to be connected, smooth, and of dimension at least $2$.

There are many characterizations of harmonic manifolds. We refer to \cite[Section 2.6]{DamekRicci} for a list of the most important ones. E.~T.~Copson and H.~S.~Ruse \cite{Copson_Ruse} proved that local harmonicity holds if and only if small geodesic spheres have constant mean curvature. It was proved by Z.~I.~Szab\'o \cite{Szabo} that in a harmonic manifold, the volume of the intersection of two geodesic balls of small radii depends only on the radii and the distance between the centers. The authors \cite{Csikos_Horvath_2gomb, Csikos_Horvath} proved  that this property characterizes harmonic manifolds even if this property is assumed only for balls of the same radius.

The present paper deals with another geometrical characterization, related to the characterization by the volume of the intersection of geodesic balls. In 1939 H. Hotelling \cite{Hotelling} showed that in the $n$-dimensional Euclidean or spherical space, the volume of a tube of small radius about a curve depends only on the length of the curve and the radius. Hotelling's result was generalized in different directions.  H. Weyl \cite{Weyl} proved that the volume of a tube of small radius about a submanifold of a Euclidean or spherical space depends only on intrinsic invariants of the submanifold and the radius. A.~Gray and L.~Vanhecke \cite{Gray-Vanhecke} extended Hotelling's theorem to rank one symmetric spaces. Answering a question asked by G.~Thorbergsson, the authors \cite{tube1} showed  that Hotelling's theorem is also true in harmonic manifolds and it holds in a symmetric space if and only if the space is flat or has rank one. This result raised the conjecture that Hotelling's theorem is true in a Riemannian manifold if and only if the manifold is harmonic. 

Our main goal is to prove the missing `only if' part of this conjecture in a stronger form.  Namely, we prove that 
	if a Riemannian manifold has the property that the volume of a tube of small radius about a \emph{geodesic segment} depends only on the radius of the tube and the length of the geodesic, then the manifold is harmonic (Theorem \ref{thm:main}). 

The above result poses the problem of finding a formula for the volume of tubes about curves in a harmonic manifold. Such formulae were computed by A.~Gray and L.~Vanhecke \cite{Gray-Vanhecke} in rank one symmetric spaces and by the authors \cite{tube1} in Damek--Ricci spaces. These computations make use of the underlying algebraic structures of symmetric and Damek--Ricci spaces, respectively, and do not give straightforward clue for the general case. However, the proof of our main theorem will provide a simple geometrical way to express the volume of tubes about curves in terms of the volume density function for all harmonic manifolds, generalizing the already known formulae.

The paper is structured as follows. In Section 2, we collect those definitions and facts that will be used in the rest of the paper. Section 3 is devoted to the proof of the main theorem. As a corollary of the theorem  and a byproduct of its proof, some further results were obtained. These include some characterizations of harmonic manifolds with the help of the volume, the total mean curvature, and the total scalar curvature of tubular hypersurfaces, collected in Section 4. We shall compute the listed intrinsic-volume-type invariants of tubular hypersurfaces for harmonic manifolds in terms of the volume density function. Our formula for the total mean curvature extends the results of A.~Gray and L.~Vanhecke \cite{Gray_Vanhecke_2} expressing the total mean curvature of tubes in rank one symmetric spaces. In Section 5, some novel characterizations of D'Atri spaces are discussed. 

\section{Preliminaries \label{sec:2}}

Let $(M,\langle\,,\rangle)$ be a Riemannian manifold. Denote by $\mathring{T}M\subseteq TM$ the domain of the exponential map of $M$, by $\exp\colon \mathring{T}M\to M$ the exponential map, and by $\exp_p \colon \mathring{T}_p M \to M$ the restriction of $\exp$ to $\mathring{T}_p M=T_pM\cap \mathring{T}M $. The injectivity radius at $p$ will be denoted by $\inj(p)$.

For $p\in M$ and $r>0$, we shall denote by $B_p(r)\subset T_pM$ and $S_p(r)\subset T_pM$ the closed ball and the sphere of radius $r$ centered at the origin $\mathbf 0_p\in T_pM$, respectively. The unit sphere $S_p(1)$ will be denoted simply by $S_p$. Denote by $SM=\bigcup_{p\in M}S_p\subset TM$ the total space of the unit sphere bundle of the tangent bundle. 

Associated to a non-zero tangent vector $\mathbf v\in T_pM\setminus\{\mathbf 0_p\}$, we shall consider the great subsphere
\[
S^{0}(\mathbf v)=\{\mathbf w\in T_pM \mid \langle \mathbf w,\mathbf v\rangle=0,\, \|\mathbf w\|=\|\mathbf v\|\}, 
\]
the hemisphere
\[
S^{+}(\mathbf v)=\{\mathbf w\in T_pM \mid \langle \mathbf w,\mathbf v\rangle\geq 0,\, \|\mathbf w\|=\|\mathbf v\|\}, 
\]
the ``disk''
\[
B^{0}(\mathbf v)=\{\mathbf w\in T_pM \mid \langle \mathbf w,\mathbf v\rangle=0,\, \|\mathbf w\|\leq\|\mathbf v\|\}, 
\]
and the half-ball
\[
B^{+}(\mathbf v)=\{\mathbf w\in T_pM \mid \langle \mathbf w,\mathbf v\rangle\geq 0,\, \|\mathbf w\|\leq\|\mathbf v\|\}. 
\]

When these sets are contained in the interior of $B_p(\inj(p))$, we can take their exponential images 
\[
\begin{matrix}
\mathcal B_p(r)=\exp(B_p(r)),&\quad&\mathcal B^0(\mathbf v)=\exp(B^0(\mathbf v)),&\quad&\mathcal B^+(\mathbf v)=\exp(B^+(\mathbf v)),
\\
\mathcal S_p(r)=\exp(S_p(r)),&\quad&\mathcal S^0(\mathbf v)=\exp(S^0(\mathbf v)),&\quad&\mathcal S^+(\mathbf v)=\exp(S^+(\mathbf v)).
\end{matrix}
\]
The sets $\mathcal B_p(r)$ and $\mathcal S_p(r)$ are the geodesic ball and sphere of radius $r$ centered at $p$, respectively. Analogously, the sets $\mathcal B^0(\mathbf v)$, $\mathcal B^+(\mathbf v)$, $\mathcal S^0(\mathbf v)$, and $\mathcal S^+(\mathbf v)$ will be called a geodesic disk, a geodesic half-ball, a geodesic great subsphere, and a geodesic hemisphere, respectively. 

Let $\omega_m=\pi^{m/2}/\Gamma(m/2+1)$ be the volume of an $m$-dimensional Euclidean unit ball.

\begin{definition}
	For a smooth injective regular curve $\gamma\colon[a,b]\to M$ and $r>0$, set 
	\[
	T(\gamma,r)=\{\mathbf v\in TM\mid \exists t\in[a,b]\text{ such that }\mathbf 
	v\in T_{\gamma(t)}M, \mathbf v\perp\gamma'(t),\text{ and }\|\mathbf v\|\leq 
	r\},\]
	and 
	\[
	P(\gamma,r)=\{\mathbf v\in TM\mid \exists t\in[a,b]\text{ such that }\mathbf 
	v\in T_{\gamma(t)}M, \mathbf v\perp\gamma'(t),\text{ and }\|\mathbf v\|=r\},\]
	Assume that $r$ is small enough to guarantee that the 
	exponential map is defined and injective on $T(\gamma,r)$. Then we define the 
	\emph{(solid) tube of radius $r$ about $\gamma$} by
	\[\mathcal T(\gamma,r)=\Exp(T(\gamma,r)),\]
	while the \emph{tubular hypersurface of radius $r$ about $\gamma$} is defined as 
	\[\mathcal P(\gamma,r)=\Exp(P(\gamma,r)).\]
\end{definition}

Denote the volume of $\mathcal T(\gamma,r)$ by $V_{\gamma}(r)$ and the $(n-1)$-dimensional volume of $\mathcal P(\gamma,r)$ by $A_{\gamma}(r)$. 

\begin{definition}
	We say that a Riemannian manifold has the \emph{tube property} if there is a function
	$v\colon [0,\infty)\to \mathbb R$ such that 
	\begin{equation}\label{eq:tube_density}
	V_{\gamma}(r)=v(r) l_{\gamma}
	\end{equation}
	for any smooth injective regular curve $\gamma\colon[a,b]\to M$ of length  $l_{\gamma}$  and 
	any sufficiently small $r$. 
	
	The manifold is said to have the \emph{tube property for geodesic curves} if equation \eqref{eq:tube_density} holds with a fixed function $v$ for all geodesic segments $\gamma$ and for any sufficiently small $r$. 
\end{definition}
The authors proved the following two theorems in \cite{tube1}.
\begin{theorem}\label{DAtri}
	A Riemannian manifold has the tube property if and only if it is a D'Atri space 
	and satisfies the tube property for geodesic curves.
\end{theorem}

\begin{theorem}\label{thm:harmonic}
	Every connected harmonic manifold has the tube property.
\end{theorem}
Our goal is to strengthen these two theorems as follows.
\begin{theorem}\label{thm:main} The following properties are equivalent for a connected Riemannian manifold:
	\begin{itemize}
		\item[\emph{(i)}] the manifold has the tube property;
		\item[\emph{(ii)}] the manifold has the tube property for geodesic curves;
		\item[\emph{(iii)}] the manifold is harmonic.
	\end{itemize}
\end{theorem}
Implication (i)$\Rightarrow$(ii) is trivial, (iii)$\Rightarrow$(i) is just Theorem \ref{thm:harmonic}, (ii)$\Rightarrow$(iii) will be proved in the next section. We note that our proof can be simplified substantially if we want to show only the implication (i)$\Rightarrow$(iii). Details of how to adapt our proof to obtain this weaker statement will be given in the remark at the end of Section 5.

The proof will use the characterization of harmonic manifolds with the help of the volume density function.
By definition, the volume density function $\theta\colon \mathring{T}M\to \mathbb R$ assigns to a tangent vector $\mathbf v\in \mathring{T}_p M$, the volume stretch factor of the derivative map $T_{\mathbf v}\exp_p\colon T_\mathbf v(T_pM)\to T_{\exp(\mathbf v)}M$. If $\mathbf e_1,\dots,\mathbf e_n$ is an orthonormal basis of $T_\mathbf v(T_pM)\cong T_pM$ with respect to the Euclidean structure $\langle\,,\rangle_p$, then $\theta(\mathbf v)=\sqrt{\det g}$, where $g$ is the $n\times n$ matrix with entries $g_{ij}=\langle T_{\mathbf v}\exp_p(\mathbf e_i),T_{\mathbf v}\exp_p(\mathbf e_j)\rangle$.

A Riemannian manifold is known to be harmonic if and only if the volume density function is radial, which means that $\theta(\mathbf v)$ depends only on $\|\mathbf v \|$ (see \cite[Section 2.6]{DamekRicci}).

For a tangent vector $\mathbf v\in TM$, $\gamma_{\mathbf v}$ will denote the maximal geodesic curve with initial velocity $\gamma_{\mathbf v}'(0)=\mathbf v$. 

The \emph{canonical geodesic involution} is the involutive diffeomorphism $\iota\colon \mathring{T}M\to \mathring{T}M$ defined by \break $\iota(\mathbf v)=-\gamma_{\mathbf v}'(1)$. 

\begin{theorem}[{\cite[Lemma 6.12]{Besse}}] \label{thm:volume_symmetry}The volume density function has the symmetry $\theta\circ\iota=\theta$.
\end{theorem}

The geodesic spray $\mathbf G\colon TM\to T(TM)$ is the vector field on the tangent bundle that generates the geodesic flow. Its value at $\mathbf v\in TM$ is the tangent vector $\mathbf G_{\mathbf v}=\frac{\ud \gamma'_{\mathbf v}(t)}{\ud t}\big|_{t=0}\in T_{\mathbf v}(TM)$. In particular, the derivative of a smooth function $f\colon TM\to \mathbb R$ with respect to the geodesic spray is given by the formula $\mathbf G_{\mathbf v}f=\frac{\ud f(\gamma'_{\mathbf v}(t))}{\ud t}\big|_{t=0}$.

The Liouville vector field $\mathbf X$ is the natural vector field on $TM$ generating the flow $\Phi_t(\mathbf v)=e^t\mathbf v$.  Denote by $\bar r\colon TM\to \mathbb R$ the function $\bar r(\mathbf v)=\|\mathbf v\|$. The normalization $\mathbf X/\bar r$ of $\mathbf X$, defined on the complement of the zero section of $TM$, will be denoted by $\partial_r$. The derivative $\partial_r f$ of a smooth function $f\colon U\to \mathbb R$ defined on an open subset $U$ of $TM$ is called the radial derivative of $f$. Its value at a nonzero tangent vector $\mathbf v\in U$ is $\partial_r f(\mathbf v)=\frac{\ud}{\ud t}f((1+t/\|\mathbf v\|)\mathbf v)|_{t=0}$.

The following lemma shows that invariance under the canonical geodesic involution is not inherited by the radial derivative of a function.
\begin{lemma} \label{le:iota_inv}If the smooth function $f$ is defined on a $\iota$-invariant open subset of $\mathring TM$ and $f\circ \iota=f$, then $(\partial_r f)\circ\iota=\partial_r f-\mathbf Gf/\bar r$. In particular, $\partial_r f$ is $\iota$-invariant if and only if $f$ is a first integral of the geodesic flow.
\end{lemma}
\begin{proof} If $\mathbf u\in SM$, and the non-zero tangent vector ${r}\mathbf u$ is in the domain of $f$, then 
\begin{align*}
(\partial_r f)\circ\iota({r}\mathbf u)&=(\partial_r f)(-{r}\gamma_{\mathbf u}'({r}))=\frac{\ud }{\ud \rho}f(-({r}+\rho)\gamma'_{\mathbf u}({r}))\big|_{\rho=0}=\frac{\ud }{\ud \rho}f\circ\iota(-({r}+\rho)\gamma'_{\mathbf u}({r}))\big|_{\rho=0}\\&=\frac{\ud }{\ud \rho}f(({r}+\rho)\gamma'_{\mathbf u}(-\rho))\big|_{\rho=0}=\partial_r f({r}\mathbf u)-\frac{\mathbf G_{{r}\mathbf u}f}{{r}}.\qedhere
\end{align*}
\end{proof}
Harmonicity can also be characterized with the help of the mean curvature of geodesic spheres. Consider the open subset ${\breve T}M=\{\mathbf v\in \mathring{T}M \mid \theta(\mathbf v)\cdot\|\mathbf v\|\neq 0\}$ in $TM$ and its intersection $\breve T_pM=\breve TM\cap T_pM$ with $T_pM$. If $\mathbf v\in \breve T_pM$, then the exponential map $\exp_p\colon \mathring T_pM\to M$ maps an open neighborhood $U\subset \breve T_pM$ of $\mathbf v$ diffeomorphically onto its image. In particular, the image of $U\cap S_p(\|\mathbf v\|)$ under the exponential map is a smooth hypersurface of $M$. By the Gauss lemma, the tangent vector $\gamma'_{\mathbf v}(1)/\|\mathbf v\|$ is a unit normal of the hypersurface $\exp(U\cap S_p(\|\mathbf v\|))$. Let $L_{\mathbf v}$ denote the shape operator of the hypersurface $\exp(U\cap S_p(\|\mathbf v\|))$ with respect to the unit normal $\gamma'_{\mathbf v}(1)/\|\mathbf v\|$ and let $h(\mathbf v)=\tr(L_{\mathbf v})$ be the trace of $L_{\mathbf v}$.

\begin{theorem}[{\cite[Theorem 3.11]{Gray_Tubes}}]\label{thm:mean_curvature}
	The mean curvature function of geodesic spheres can be expressed with the help of the volume density function as follows
	\[
	h=-\partial_r(\ln(\bar r^{n-1}\theta))=-\frac{n-1}{\bar r}-\frac{\partial_r\theta}{\theta}.
	\]
\end{theorem}

We shall use the big $O$ notation in the following sense. For two functions $F$ and $G$ defined on a neighborhood of the zero section of $TM$, we write $F(r\mathbf u)=G(r\mathbf u)+O(r^{m})$ if for any compact subset $K\subseteq M$, there exist positive numbers $C_K>0$ and $r_K>0$ such that $|F(r\mathbf u)-G(r\mathbf u)|\leq C_Kr^m$ for any $\mathbf u\in SM$ with base point in $K$, and any $r$ for which $|r|<r_K$. Typically we shall apply this notation in the case when $M$ is a real analytic manifold and $F-G$ is analytic in a neighborhood of the zero section of $TM$. In that case, we can choose a chart $\phi=(x^1,\dots,x^n)$ from the $\mathcal C^{\omega}$-atlas of $M$ around any point $p\in M$ mapping $p$ to $\phi(p)=\mathbf 0$. The chart $\phi$ induces a chart $(\hat x^1,\dots,\hat x^n,y^1,\dots,y^n)=(x^1\circ \pi,\dots, x^n\circ \pi,\ud x^1,\dots,\ud x^n)$ on the tangent bundle, where $\pi\colon TM\to M$ is the projection of the bundle. The difference $F-G$ can be written as the sum of a power series $F-G=\sum_{\alpha,\beta} a_{\alpha,\beta}\hat x^{\alpha}y^{\beta}$ around $\mathbf 0_p\in T_pM$, where the sum goes for all multiindices $\alpha$ and $\beta$. Then equation $F(r\mathbf u)=G(r\mathbf u)+O(r^{m})$ is equivalent to the condition that $a_{\alpha,\beta}=0$ for all $\alpha,\beta$ such that $|\beta|<m$, for any choice of $p$ and $\phi$.

\section{Proof of Theorem \ref{thm:main}\label{sec:3}}

In this section, we prove that if a Riemannian manifold $M$ satisfies equation \eqref{eq:tube_density} for any geodesic segment $\gamma$ and any sufficiently small radius $r$ with a given function $v$, then $M$ is harmonic.

It was proved in \cite{tube1} that the tube property implies that the manifold is Einstein. As the proof used the tube property only for geodesic curves, we obtain that $M$ must be an Einstein manifold, thus, by the Kazdan--DeTurck theorem \cite{Kazdan_DeTurck}, normal coordinate charts give a $\mathcal C^{\omega}$-atlas on $M$ with respect to which the metric tensor  is real analytic.

As $M$ is a real analytic Riemannian manifold, $SM$ is also a real analytic manifold, and there exist analytic functions $a_i\colon SM\to \mathbb R$ for $i=0,1,2,\dots$, so that 
\begin{equation}\label{eq:power_series}
\theta(r\mathbf u)=\sum_{i=0}^{\infty}a_i(\mathbf u)r^i
\end{equation}
for any $\mathbf u\in SM$ and any $r$ with sufficiently small absolute value. The identity $\theta(r\mathbf u)=\theta((-r)(-\mathbf u))$ gives at once that 
\begin{equation}\label{eq:symmetry_of_a_i}
a_i(-\mathbf u)=(-1)^ia_i(\mathbf u).
\end{equation}

Equation \eqref{eq:power_series} shows that $M$ is harmonic if and only if all the functions $a_i$ are constant. Observe that if $a_i$ is constant for an odd $i$, then $a_i\equiv 0$ by equation \eqref{eq:symmetry_of_a_i}.

\begin{definition} We shall say that a real analytic Riemannian manifold $M$ is \emph{harmonic up to order $k$} if the functions $a_i$ are constant for $0\leq i\leq k$.
\end{definition}
When $a_i$ is a constant function, we shall write $a_i$ instead of $a_i(\mathbf u)$, whatever the unit tangent vector $\mathbf u\in SM$ is, and think of $a_i$ as a real number.

We are going to prove that $M$ is harmonic up to order $2k$ by induction on $k$. It is well-known that $a_0\equiv 1$ (\cite[Corollary 9.9]{Gray_Tubes}) for any Riemannian manifold, so the base case $k=0$ is settled. Assume that $M$ is harmonic up to order $2k$.

\begin{lemma}\label{le:Vanhecke} If a real analytic manifold is harmonic up to order $2k$, then it is also harmonic up to order $2k+1$. 
\end{lemma}
\begin{proof} L.~Vanhecke formulated and proved a slight modification of this statement in \cite{Vanhecke}, but in fact his proof implies our lemma as well. We recall Vanhecke's proof for the reader's convenience.
Choose an arbitrary unit tangent vector $\mathbf u\in SM$.  Then Theorem \ref{thm:volume_symmetry} implies
\begin{equation*}
\sum_{i=0}^{\infty}a_i(\mathbf u)r^i=\theta(r\mathbf u)=\theta(-r\gamma'_{\mathbf u}(r))=\sum_{i=0}^{\infty}a_i(\gamma'_{\mathbf u}(r))(-r)^i.
\end{equation*}
Writing the analytical functions $\alpha_i(r)=a_i(\gamma'_{\mathbf u}(r))$ as the sum of their Taylor series, we obtain
\begin{equation*}
\sum_{i=0}^{\infty}a_i(\mathbf u)r^i=\sum_{i=0}^{\infty}\sum_{j=0}^{\infty}(-1)^i\frac{\alpha_i^{(j)}(0)}{j!}r^{i+j}.
\end{equation*}
Since $\alpha_i(0)=a_i(\mathbf u)$, equating the coefficients of $r^{2k+1}$ on the two sides gives
\begin{equation}\label{eq:Vanhecke}
2a_{2k+1}(\mathbf u)=\sum_{j=1}^{2k+1}(-1)^{2k+1-j}\frac{\alpha_{2k+1-j}^{(j)}(0)}{j!}.
\end{equation}
As the functions $\alpha_{2k+1-j}$ are constant for $1\leq j\leq 2k+1$ by our assumption, their derivatives vanish, therefore $a_{2k+1}\equiv 0$.
%
%
\end{proof}

A Riemannian manifold is a D'Atri space if and only if the mean curvature function of geodesic spheres satisfies the symmetry relation $h=h\circ\iota$ (see \cite[Section 2.7]{DamekRicci}). The following lemma is a ``stability version'' of the fact that harmonic manifolds are D'Atri spaces.

\begin{lemma}\label{le:stability_D'Atri} If $M$ is harmonic up to order $2k$, then 
\[h(r\mathbf u)-h(\iota(r\mathbf u))=O(r^{2k+2}).\]
\end{lemma}
\begin{proof}
Using Lemma \ref{le:Vanhecke}, we can write $\theta$ as
\[\theta(r\mathbf u)=p(r)+a_{2k+2}(\mathbf u)r^{2k+2}+O(r^{2k+3}),\]
where $p$  is a polynomial of degree $2k$ with $p(r)=\sum_{i=0}^{2k}a_ir^i=1+O(r^2)$. 
Thus, $p^{-1}(r)=1+O(r^2)$, and 
\begin{equation}\label{eq:ln_theta}
\ln(\theta(r\mathbf u))=\ln(p(r))+\ln\big(1+p^{-1}(r)(a_{2k+2}(\mathbf u)r^{2k+2}+O(r^{2k+3}))\big)=\ln(p(r))+a_{2k+2}(\mathbf u)r^{2k+2}+O(r^{2k+3}).
\end{equation}
If an analytic function $F$ defined on an open neighborhood of the zero section of $TM$ can be written as 
\[
F(r\mathbf u)=\sum_{i=1}^{\infty}b_i(\mathbf u)r^i
\]
for $\mathbf u\in SM$ and $r>0$, then 
\[
\frac{\mathbf G_{r\mathbf u}F}{r}=\sum_{i=1}^{\infty}(\mathbf G_{\mathbf u}b_i)r^i.
\]
Applying this statement to the series \eqref{eq:ln_theta} of the $\iota$-invariant function $\ln\circ \theta$, and using Theorem \ref{thm:mean_curvature} and Lemma \ref{le:iota_inv}, we obtain
\[
h(r\mathbf u)-h(\iota(r\mathbf u))=\partial_r(\ln\circ\theta)(\iota(r\mathbf u))-\partial_r(\ln\circ\theta)(r\mathbf u)=-(\mathbf G_{\mathbf u}a_{2k+2})r^{2k+2}+O(r^{2k+3})=O(r^{2k+2}).\qedhere
\]  
\end{proof}

We shall need the following Fubini-type formula.
\begin{lemma}\label{le:Fubini} For any continuous function $f\colon V_2(\mathbb R^n)\to\mathbb R$ defined on the Stiefel manifold
\[V_2(\mathbb R^n)=\{(\mathbf u,\mathbf v)\in \mathbb S^{n-1}\times \mathbb S^{n-1}\mid \mathbf u\perp\mathbf v\}\cong \mathrm{O}(n)/\mathrm{O}(n-2),
\] 
we have
\begin{equation}\label{eq:Stiefel-Fubini}
\int_{\mathbb S^{n-1}}\left(\int_{\mathbb S^{n-2}_{\mathbf u}}f(\mathbf u,\mathbf v)\id \mathbf v\right)\id \mathbf u=\int_{\mathbb S^{n-1}}\left(\int_{\mathbb S^{n-2}_{\mathbf v}}f(\mathbf u,\mathbf v)\id \mathbf u\right)\id \mathbf v,
\end{equation}
where $\mathbb S^{n-1}$ is the sphere of unit vectors in $\mathbb R^n$, $\mathbb S^{n-2}_{\mathbf u}$ denotes the sphere of unit vectors  orthogonal to $\mathbf u$, integrations over the unit spheres are always taken with respect to the volume measures induced by the standard Riemannian metric on them.
\end{lemma}
\begin{proof} Both sides of \eqref{eq:Stiefel-Fubini} define an $\mathrm{O}(n)$-invariant positive linear functional on the space of continuous functions on $V_2(\mathbb R^n)$, thus, by the Riesz representation theorem, there exist unique regular Borel measures $\mu$ and $\nu$ on $V_2(\mathbb R^n)$ such that for any continuous function $f\colon V_2(\mathbb R^n)\to \mathbb R$, we have
\begin{equation*}
\int_{\mathbb S^{n-1}}\left(\int_{\mathbb S^{n-2}_{\mathbf u}}f(\mathbf u,\mathbf v)\id \mathbf v\right)\id \mathbf u=\int_{V_2(\mathbb R^n)}f\id\mu\quad\text{ and }\quad\int_{\mathbb S^{n-1}}\left(\int_{\mathbb S^{n-2}_{\mathbf v}}f(\mathbf u,\mathbf v)\id \mathbf u\right)\id \mathbf v=\int_{V_2(\mathbb R^n)}f\id\nu.
\end{equation*}
Both $\mu$ and $\nu$ are $\mathrm{O}(n)$-invariant, and normalized so that $\mu(V_2(\mathbb R^n))=\nu(V_2(\mathbb R^n))=n(n-1)\omega_n\omega_{n-1}$, hence $\mu=\nu$. The equality of the measures implies the lemma.
\end{proof}
Our aim is to show that the function $a_{2k+2}$ is constant. As an intermediate step, we prove a weaker statement.

\begin{lemma}\label{ball_homogeneity}
	Assume that a Riemannian manifold $M$ having the tube property for geodesic curves is harmonic up to order $2k$. Then the integral $\int_{S_p}a_{2k+2}(\mathbf u)\id \mathbf u$ does not depend on $p\in M$.
\end{lemma}
\begin{proof} Fix a point $p\in M$ and choose an arbitrary unit tangent vector $\mathbf u\in S_p$. Consider the tube $\mathcal T(\gamma,r)$ about the geodesic segment  $\gamma=\gamma_{\mathbf u}|_{[0,l]}$. Using equation \eqref{eq:tube_density} and the results of L.~Vanhecke and T.~J.~Willmore \cite[equations (3.10) and (3.37)]{Vanhecke-Willmore}, the volume of $\mathcal T(\gamma,r)$ can be written as 
\[
V_{\gamma}(r)=v(r)l=-\int_0^{l}\int_0^r \rho^{n-1}\int_{S^0(\gamma'(t))}\langle L_{\iota(\rho\mathbf v)}(\gamma'(t)),\gamma'(t)\rangle \theta(\rho\mathbf v)\id \mathbf v\id \rho\id t.
\]
Differentiating this equation with respect to $l$ at $l=0$ and with respect to $r$, we obtain
\begin{equation}\label{eq:hbh1}
v'(r)=-\int_{S^0(\mathbf u)}\langle L_{\iota(r \mathbf v)}(\mathbf u),\mathbf u\rangle r^{n-1}\theta(r\mathbf v)\id \mathbf v.
\end{equation}
Integrating equation \eqref{eq:hbh1} as $\mathbf u$ is running over the unit sphere $S_p$, and applying first Lemma \ref{le:Fubini}  then Lemma \ref{le:stability_D'Atri} and Theorem \ref{thm:mean_curvature},  we obtain
\begin{align*}
	{n}\omega_{n}v'(r)&= -\int_{S_p}\int_{S^{0}(\mathbf u)}\langle L_{\iota(r \mathbf v)}(\mathbf u),\mathbf u\rangle r^{n-1} \theta(r\mathbf v)\id \mathbf v\id \mathbf u=-\int_{S_p}\int_{S^{0}(\mathbf v)}\langle L_{\iota(r \mathbf v)}(\mathbf u),\mathbf u\rangle r^{n-1}\theta(r\mathbf v)\id \mathbf u\id \mathbf v
\\&=-\omega_{n-1}\int_{S_p}h(\iota(r \mathbf v)) r^{n-1}\theta(r\mathbf v)\id \mathbf v=-\omega_{n-1}\int_{S_p}h(r \mathbf v) r^{n-1}\theta(r\mathbf v)\id \mathbf v+O(r^{2k+n+1})
\\&= \omega_{n-1}\int_{S_p}\partial_r(r^{n-1}\theta(r\mathbf v))\id \mathbf v+O(r^{2k+n+1}).
	\end{align*}
Integration with respect to $r$ yields
\begin{equation}\label{eq:surf_geod_sphere}
\frac{{n}\omega_{n}}{\omega_{n-1}}v(r)=
 \int_{S_p}r^{n-1}\theta(r\mathbf v)\id \mathbf v+O(r^{2k+n+2})=\sum_{i=0}^{2k+2}\left(\int_{S_p}a_i(\mathbf v)\id \mathbf v\right)r^{n+i-1}+O(r^{2k+n+2}).
\end{equation}
From this, we get 
\[
\int_{S_p}a_{2k+2}(\mathbf v)\id \mathbf v=\frac{{n}\omega_{n}}{\omega_{n-1}}\cdot \frac{v^{(2k+n+1)}(0)}{(2k+n+1)!},
\]
where the right hand side does not depend on $p$.
\end{proof}

\begin{lemma}\label{half-ball}
If $M$ is harmonic up to order $2k$ and satisfies the tube property for geodesic curves with function $v$, then for $\mathbf u\in S_p$ and $r\in(0,\inj(p))$, the volume of the geodesic half-ball $\mathcal B^+(r\mathbf u)$ can be expressed as 
\begin{equation*}\label{eq:half_ball_volume}
\vol_{n}(\mathcal B^+(r\mathbf u))=\frac{{n}\omega_{n}}{2\omega_{n-1}}\int_0^r v(\rho)\id \rho +O(r^{2k+n+3}).
\end{equation*}
\end{lemma}
\begin{proof}
The volume of the half-ball $\mathcal B^+(r\mathbf u)$ is expressed by the integral
\[
\vol_n(\mathcal B^+(r\mathbf u))=\int_0^r\int_{S^+(\mathbf u)}\rho^{n-1}\theta(\rho\mathbf v)\id \mathbf v\id \rho=\sum_{i=0}^{k+1}\left(\int_{S^+(\mathbf u)}a_{2i}(\mathbf v)\id \mathbf v\right)\frac{r^{n+2i}}{n+2i}+O(r^{2k+n+3}).
\]
Since the functions $a_{2i}$ are even, this implies that 
\[
\vol_n(\mathcal B^+(r\mathbf u))=\vol_n(\mathcal B^+(-r\mathbf u))+O(r^{2k+n+3}).
\]

As the opposite half-balls $\mathcal B^+(\pm r\mathbf u)$ have equal volumes up to an error term $O(r^{2k+n+3})$, their volumes are equal to half the volume of the geodesic ball $\mathcal B_p(r)$ up to an error of the same order. The latter volume can be obtained by integrating \eqref{eq:surf_geod_sphere}, thus,
\begin{equation*}\label{eq:ball_volume}
\frac{{n}\omega_{n}}{\omega_{n-1}}\int_0^r v(\rho)\id \rho
+O(r^{2k+n+3})=\int_0^r\int_{S_p}\rho^{n-1}\theta(\rho\mathbf v)\id \mathbf v\id \rho=\vol_n(\mathcal B_p(r)).\qedhere
\end{equation*}
\end{proof}

Now we are ready to complete the induction step. 

For $\mathbf v\in \mathring{T}_pM$, let $\Pi_{\mathbf v}\colon T_pM\to T_{\exp(\mathbf v)}M$ denote the parallel transport map along the geodesic $\gamma_{\mathbf v}$.
Let $\mathbf u\in S_p$ be an arbitrary unit tangent vector and consider the geodesic curve $\gamma=\gamma_{\mathbf u}$. Consider the isotopy 
$\Phi_{t}=\exp_{\gamma(t)}\circ\Pi_{t\mathbf u} \circ\exp_p^{-1}$ defined on a small 
neighborhood of $p$ for small values of $t$. Let $X(q)=\left.\frac{\ud}{\ud 
t}\Phi_t(q)\right|_{t=0}$ be the initial velocity vector field of the isotopy. 

It is clear from the definitions that the image of $\mathcal B^+(r\mathbf u)$ under $\Phi_t$ is the half-ball $\mathcal B^+(r\gamma'(t))$ for small $r>0$ and $t$. The boundary  $\partial \mathcal B^+(r\mathbf u)$ of the half-ball is the union of the geodesic disk $\mathcal B^0(r\mathbf u)$ and the geodesic hemisphere $\mathcal S^+(r\mathbf u)$. 

The outer unit normal vector field $N$ of $\mathcal B^+(r\mathbf u)$ is a smooth vector field on the smooth part $\partial \mathcal B^+(r\mathbf u)\setminus \mathcal S^0(r\mathbf u)$ of the boundary $\partial \mathcal B^+(r\mathbf u)$, thus, although it does not extend continuously to $\mathcal S^0(r\mathbf u)$ it is defined almost everywhere. Then, by Lemma \ref{half-ball}, we have
\begin{equation}\label{eq:felgombvar}
\int_{\mathcal B^0(r\mathbf u)}\langle X,N \rangle\id\sigma+\int_{\mathcal S^+(r\mathbf u)}\langle X,N \rangle\id\sigma=\int_{\partial \mathcal B^+(r\mathbf u)}\langle X,N \rangle\id\sigma=\frac{\ud}{\ud t}\vol_n(\mathcal B^+(r\gamma'(t)))\Big|_{t=0}=O(r^{2k+n+3}).
\end{equation}

The images $\Phi_t(\mathcal B^0(r\mathbf u))$ of the geodesic disk $\mathcal B^0(r\mathbf u)$ under the isotopy sweep out the tube about $\gamma$, i.e., 
\[\mathcal T(\gamma|_{[0,\varepsilon]},r)=\bigcup_{t\in[0,\varepsilon]}\Phi_t(\mathcal B^0(r\mathbf u)).\]
Differentiating the volume of this tube with respect to $\varepsilon$ at $\varepsilon=0$, and using equation \eqref{eq:felgombvar}, we obtain
\begin{equation}\label{eq:bullet1}
v(r)=\int_{\mathcal B^0(r\mathbf u)}\langle X,-N \rangle\id\sigma=\int_{\mathcal S^+(r\mathbf u)}\langle X,N \rangle\id\sigma+O(r^{2k+n+3}).
\end{equation} 

For a unit tangent vector  $\mathbf v\in S^+(\mathbf u)$, the map $\Gamma\colon [0,r]\times [-\varepsilon,\varepsilon]\to M$, 
$\Gamma(s,t)=\Phi_t(\gamma_{\mathbf v}(s))$ is a geodesic variation of $\gamma_{\mathbf v}$. Thus, $J(s)=X(\gamma_{\mathbf v}(s))$ is a 
Jacobi field along $\gamma_{\mathbf v}$ with initial values $J(0)=\mathbf u$ and 
$J'(0)=\mathbf 0_p$. The tangential component $\langle J,\gamma_{\mathbf v}'\rangle$ of $J$ must be a linear function of the form $\langle J(t),\gamma_{\mathbf v}'(t)\rangle=at+b$ for some constants $a,b\in \mathbb R$. The initial conditions for $J$ give $a=0$ and $b=\langle \mathbf u,\gamma_{\mathbf v}'(0)\rangle=\langle \mathbf u,\mathbf v\rangle$. As $\gamma'_{\mathbf v}(r)=N(\exp(r\mathbf v))$ by the Gauss lemma, we have
\begin{equation}\label{eq:Jacobi_Gauss}
\langle X(\exp(r\mathbf v)), N(\exp(r\mathbf v))\rangle=\langle J(r), \gamma'_{\mathbf v}(r)\rangle=b=\langle \mathbf u,\mathbf v\rangle.
\end{equation} 

Therefore, equations \eqref{eq:bullet1} and \eqref{eq:Jacobi_Gauss} yield
\begin{equation}\label{eq:bullet2}
v(r)=\int_{S^+(\mathbf u)}\langle\mathbf u,\mathbf v\rangle\theta(r\mathbf v)r^{n-1}\id \mathbf v+O(r^{2k+n+3}).
\end{equation}
The coefficients of $r^{2k+n+1}$ of the Taylor series of the two sides must be equal, thus, using the symmetry \eqref{eq:symmetry_of_a_i}, we obtain
\begin{equation}\label{FH}
2\frac{v^{(2k+n+1)}(0)}{(2k+n+1)!}=\int_{S_p}|\langle\mathbf u,\mathbf v\rangle|a_{2k+2}(\mathbf v)\id \mathbf v.
\end{equation}

We recall some facts about the cosine transform. See 
\cite{Groemer} for more details.
The cosine transform $\CT$ is the integral transform
$\CT \colon\mathcal C^0(\mathbb S^{n-1})\to\mathcal C^0(\mathbb 
S^{n-1})$ defined by
\[(\CT(f))(\mathbf u)=\int_{\mathbb S^{n-1}}|\langle\mathbf u,\mathbf 
v\rangle|f(\mathbf v)\id \mathbf v.\]
It is known that $\CT(g)=0$ if and only if $g$ is odd 
(see \cite[Proposition 3.4.10]{Groemer}).
We also have that the cosine transform of the constant $1$ function is
\[(\CT(1))(\mathbf u)=\int_{\mathbb S^{n-1}}|\langle\mathbf u,\mathbf 
v\rangle|\id\mathbf v=2\omega_{n-1},\]
(see \cite[Lemma 3.4.5]{Groemer}).

The above facts and formula \eqref{FH} imply that the function
\[f(\mathbf v)=a_{2k+2}(\mathbf v)-\frac{v^{(2k+n+1)}(0)}{\omega_{n-1}(2k+n+1)!}\]
is an odd function. On the other hand, as $a_{2k+2}$ is an even function, so is $f$. Therefore, $f$ must vanish, and 
\[a_{2k+2}(\mathbf v)=\frac{v^{(2k+n+1)}(0)}{\omega_{n-1}(2k+n+1)!}\]
is a constant function. This completes the induction step and the proof of Theorem \ref{thm:main}.

\begin{corollary}
	In a harmonic manifold, the volume of solid tubes of small radius $r$ about a curve $\gamma$ is given by equation \eqref{eq:tube_density} with the function
	\[v(r)=\omega_{n-1}r^{n-1}\underline{\theta}(r)=\frac{\omega_{n-1}}{n\omega_n}\vol_{n-1}(\mathcal{S}_p(r)),\]
	where $\underline{\theta}\colon [0,\infty)\to\mathbb R$ is a function for which $\theta=\underline{\theta}\circ\bar r$, and $p$ is an arbitrary point of $M$.
\end{corollary}
\begin{proof} Indeed, as \eqref{eq:bullet2} holds for every $k$, we have 
\begin{equation*}\label{eq:bullet3}
v(r)=\int_{S^+(\mathbf u)}\langle\mathbf u,\mathbf v\rangle\theta(r\mathbf v)r^{n-1}\id \mathbf v=\frac12 r^{n-1}\underline{\theta}(r)\CT(1)=\omega_{n-1}r^{n-1}\underline{\theta}(r).\qedhere
\end{equation*}
\end{proof}

\begin{rem}
A.~Gray and L.~Vanhecke \cite{Gray_Vanhecke_balls} posed the general question to what extent a Riemannian manifold is determined by the volumes of small geodesic balls.  An analogue of this question for the volumes of tubes about curves was studied by them in \cite{Gray-Vanhecke}. They proved that some families of Riemannian manifolds can be characterized by the volumes of their geodesic balls or tubes. For example, they proved in  \cite{Gray_Vanhecke_balls}, that if, in a Riemannian manifold, the volume of geodesic balls of small radius $r$  equals $\omega_nr^n$, then the manifold is flat. In \cite[Section 7]{Gray-Vanhecke} they show some characterizations of flat and rank one symmetric spaces in terms of the volumes of the tubes about curves. As the volumes of geodesic balls and volumes of tubes of small radius can be expressed with the help of the volume density function in harmonic manifolds, these results suggest the question whether a harmonic manifold is determined up to local isometry by its volume density function $\theta$. The answer to this question is negative, and counterexamples can be found among the Damek--Ricci spaces. Damek--Ricci spaces are solvable Lie groups with a left invariant Riemannian metric, the Lie algebra of which is a one dimensional extension of a generalized Heisenberg Lie algebra $\mathfrak n=\mathfrak v\oplus\mathfrak z$ with center $\mathfrak{z}$. It is known (\cite[Theorem 1]{Damek_Ricci}) that the dimensions $p=\dim \mathfrak v$ and $q=\dim \mathfrak z$ determine the volume density function of the Damek--Ricci space by $\theta(\mathbf v)=\underline{\theta}(\|\mathbf v\|)$, where \[\underline{\theta}(r)=\cosh^q(r/2)\left(\frac{\sinh(r/2)}{r/2}\right)^{p+q}.\]
However, if $q\equiv 3$ (mod~$4$), then the numbers $p$ and $q$ do not determine the isomorphism class of the Lie algebra $\mathfrak n$ uniquely, and the corresponding Damek--Ricci spaces are not necessarily isometric. An explicit example  when this happens is given in \cite[Section 4.1.9, Theorem, (ii)]{DamekRicci}: There are two non-isomorphic generalized Heisenberg Lie algebras with dimension parameters $p=8$ and $q=3$, and the corresponding  Damek--Ricci spaces are not isometric, because one of them is the quaternionic hyperbolic space, hence symmetric,  while the other is not symmetric.
\end{rem}

\section{Total mean curvature and total scalar curvature of tubular hypersurfaces}

When the volume of the tube $\mathcal T(\gamma,r)$ depends only on the radius $r$ and the length of the curve $\gamma$ for small $r$, then the same is true for the derivatives of this volume function with respect to the radius. The derivatives are geometric invariants of the tubular hypersurface $\mathcal{P}(\gamma,r)$. As $\mathcal{P}(\gamma,r+\Delta)$ is a parallel hypersurface of $\mathcal{P}(\gamma,r)$ lying at distance $\Delta$ away from it, we can apply the Steiner-type formula of E.~Abbena, A.~Gray, and L.~Vanhecke \cite[Theorem 3.5]{Abbena_Gray_Vanhecke} to compute the power series of the volume of $\mathcal T(\gamma,r+\Delta)$ with respect to $\Delta$ for fixed $\gamma$ and $r$. The formula gives
\begin{equation}\label{eq:Steiner}
V_{\gamma}(r+\Delta)=V_{\gamma}(r)+A_{\gamma}(r)\Delta+\int_{\mathcal P(\gamma,r)}\left[-\mu^P(p)\frac{\Delta^2}{2} +\big(\rho(N(p))+\tau^P(p)-\tau(p)\big)\frac{\Delta^3}{6}\right]\id p+O(\Delta^4),
\end{equation}
where $\mu^P(p)$ is the sum of the principal curvatures of $\mathcal P(\gamma,r)$ at $p\in \mathcal P(\gamma,r)$ with respect to the outer unit normal $N(p)$ of the solid tube $\mathcal T(\gamma,r)$ at its boundary point $p$, $\rho(N(p))=\mathrm{Ric}\,(N(p),N(p))$ is the Ricci curvature of $M$ in the direction $N(p)$, $\tau(p)$ and $\tau^P(p)$ are the scalar curvatures of $M$ and $\mathcal P(\gamma,r)$ at $p$, respectively.

Since we have $V_{\gamma}(r)=\omega_{n-1}r^{n-1}l_{\gamma}+O(r^n)$ for any fixed curve $\gamma$, the volume function $V_{\gamma}$  is uniquely determined by its $k$th derivative $V^{(k)}_{\gamma}$ if $k\leq n$. Thus, the following statement is a straightforward consequence of Theorem \ref{thm:main}.

\begin{corollary}
For a connected Riemannian manifold $M$, the following properties are equivalent:
\begin{itemize}
	\item[\emph{(i)}] $M$ is harmonic.
	\item[\emph{(ii)}] For any regularly parameterized simple arc $\gamma$, the $(n-1)$-dimensional volume $A_{\gamma}(r)$ of the tubular hypersurface $\mathcal P(\gamma,r)$ depends only on $r$ and the length $l_{\gamma}$ of $\gamma$.
	\item[\emph{(iii)}] For any geodesic segment $\gamma$, the volume $A_{\gamma}(r)$ depends only on $r$ and the length of $\gamma$.
	\item[\emph{(iv)}] For any regularly parameterized simple arc $\gamma$, the total mean curvature \[H_{\gamma}(r)=\frac{1}{n-1}\int_{\mathcal P(\gamma,r)}\mu^P(p)\id p\] of $\mathcal P(\gamma,r)$ depends only on $r$ and the length of $\gamma$.
	\item[\emph{(v)}]  For any geodesic segment $\gamma$, the total mean curvature $H_{\gamma}(r)$ depends only on $r$ and the length of $\gamma$.
\end{itemize} 
If $M$ is harmonic, then the volume and the total mean curvature of the tubular hypersurface can be expressed as
\begin{equation}\label{eq:surface}
A_{\gamma}(r)=v'(r)l_{\gamma}=\omega_{n-1}(r^{n-1}\underline{\theta}'(r)+(n-1)r^{n-2}\underline{\theta}(r))l_{\gamma},
\end{equation}
and 
\[
H_{\gamma}(r)=-\frac{v''(r)}{n-1}l_{\gamma}=-\omega_{n-1}\left(\frac{r^{n-1}}{n-1}\underline{\theta}''(r)+2r^{n-2}\underline{\theta}'(r)+(n-2)r^{n-3}\underline{\theta}(r)\right)l_{\gamma}.
\]
\end{corollary}
The total scalar curvature of the tubular surface appears in the third derivative of the volume of $\mathcal T(\gamma,r)$ with respect to $r$. This observation enables us to prove the following theorem.
\begin{theorem} If $M$ is harmonic, then the total scalar curvature of the tubular hypersurface $\mathcal P(\gamma,r)$ about a regularly parameterized simple arc $\gamma$ is given by 
	\begin{equation}\label{eq:total_scalar_curvature}
	\int_{\mathcal P(\gamma,r)}\tau^P(p)\id p=\left(v'''(r)-3(n-1)\underline{\theta}''(0)v'(r)\right)l_{\gamma}
	\end{equation}
	for small values of $r$, hence it depends only on $r$ and the length $l_{\gamma}$ of $\gamma$. 
	
	Conversely, if $\dim M>3$, and for all geodesic segments $\gamma$,  the total scalar curvature of $\mathcal P(\gamma,r)$ depends only on the length of $\gamma$ and $r$, then the manifold is harmonic.
\end{theorem}
\begin{proof} Assume first that $M$ is harmonic. Then the third derivative of $V_{\gamma}=vl_{\gamma}$ at $r$ equals
\begin{equation}\label{eq:3rd_derivative}
v'''(r)l_{\gamma}=\int_{\mathcal P(\gamma,r)}\big(\rho(N(p))+\tau^P(p)-\tau(p)\big)\id p
\end{equation}
by equation \eqref{eq:Steiner}. Harmonic manifolds are Einstein manifolds with constant Ricci curvature $\rho=-3\underline{\theta}''(0)$   and scalar curvature $\tau=n\rho$ by \cite[Corollary 9.9]{Gray_Tubes}. Expressing the total scalar curvature of $\mathcal P(\gamma,r)$ from \eqref{eq:3rd_derivative} and using \eqref{eq:surface}, we obtain equation \eqref{eq:total_scalar_curvature}. 

For the second part of the theorem, assume that $\dim M>3$ and $M$ has the property that the total scalar curvature $C_{\gamma}(r)$ of the tubular hypersurface $\mathcal P(\gamma,r)$ depends only on $r$ and the length $l_{\gamma}$ if $\gamma$ is a geodesic segment and $r$ is small. Since the total scalar curvature is obviously an additive function of $l_{\gamma}$, it must have the form $C_{\gamma}(r)=c(r)l_{\gamma}$. The initial terms of the Taylor series of $C_{\gamma}(r)$ were computed by L.~Gheysens and L.~Vanhecke \cite{Gheysens_Vanhecke}. According to \cite[Theorem 5.1]{Gheysens_Vanhecke}, if $\gamma\colon [a,b]\to M$ is a fixed injective unit speed curve, then for small radii we have
\[
C_{\gamma}(r)=(n-1)\omega_{n-1}r^{n-4}\int_a^b\left[(n-2)(n-3)-\frac{n-3}{6(n-1)}\{(n-4)\tau(\gamma(t))+(n+2)\rho(\gamma'(t))\}r^2+O(r^4)\right]\id t.
\] 
Using the special form of $C_{\gamma}(r)$,  we get from this equation that
\[
\lim_{r\to 0}\frac{6\big((n-1)(n-2)(n-3)\omega_{n-1}r^{n-4}-c(r)\big)}{(n-3)\omega_{n-1}r^{n-2}}=(n-4)\tau(\gamma(t))+(n+2)\rho(\gamma'(t))
\]
holds for any $t\in[a,b]$. Since $\gamma(t)$ can be any point $p$ of $M$, and $\gamma'(t)=\mathbf u$ can be any unit tangent vector at $p$, we obtain that the quantity $(n-4)\tau(p)+(n+2)\rho(\mathbf u)$ does not depend on the choice of $p\in M$ and the unit tangent vector $\mathbf u\in S_p$. This implies that $M$ is an Einstein manifold with constant Ricci curvature equal to
\[
\rho=\lim_{r\to 0}\frac{6\big((n-1)(n-2)(n-3)\omega_{n-1}r^{n-4}-c(r)\big)}{(n-1)(n-2)(n-3)\omega_{n-1}r^{n-2}}.
\]

Differentiating \eqref{eq:Steiner} three times with respect to $\Delta$ at $\Delta=0$ and using the fact that $M$ is Einstein, we obtain that the function $V_{\gamma}$ satisfies the differential equation
\begin{equation}\label{eq:ode_for_V}
V'''_{\gamma}(r)=(1-n)\rho V'_{\gamma}(r)+c(r)l_{\gamma}
\end{equation}
for small positive values of $r$, for any geodesic segment $\gamma$.
Let $v$ be the unique non-extendable solution of the differential equation 
\begin{equation}\label{eq:ode_for_v}
 v'''=(1-n)\rho v'+c
\end{equation}
with initial condition $ v(0)=v'(0)= v''(0)=0$. Since $V_{\gamma}(0)=V'_{\gamma}(0)=V_{\gamma}''(0)=0$, comparing the differential equations \eqref{eq:ode_for_V} and \eqref{eq:ode_for_v}, we see that $ V_{\gamma}(r)=v(r)l_{\gamma}$ must hold for any geodesic segment $\gamma$ and any sufficiently small $r$, hence $M$ is harmonic by Theorem \ref{thm:main}.
\end{proof}
\begin{rem}
	We do not know whether condition $\dim M>3$ can be dropped in the second half of the theorem. As for the specialities of the case of 3-dimensional manifolds, see the comments in \cite{Gheysens_Vanhecke} following the proof of Theorem 5.4.
\end{rem}
\section{Some characterizations of D'Atri spaces}
Ball-homogeneous spaces were introduced by O.~Kowalski and L.~Vanhecke \cite{Kowalski_Vanhecke} as Riemannian manifolds in which the volume of geodesic balls of small radii depends only on the radius of the ball. Analogously, we say that a Riemannian manifold is \emph{half-ball homogeneous} if the volume of a geodesic half-ball of small radius depends only on the radius of the half-ball.
\begin{theorem}\label{thm:D'Atri}
	The following conditions for a Riemannian manifold $M$ are equivalent:
	\begin{itemize}
		\item[\emph{(i)}] The manifold is a D'Atri space.
		\item[\emph{(ii)}] The manifold is half-ball homogeneous.
		\item[\emph{(iii)}] The manifold is real analytic, and the functions $t\mapsto a_k(\gamma'(t))$,  are constant for every $k$ and every unit speed geodesic curve $\gamma$, where the functions $a_k\colon SM\to \mathbb R$ are defined by equation \eqref{eq:power_series}.
		\item[\emph{(iv)}]  The volume density function $\theta$ is a first integral of the geodesic flow.
	\end{itemize}
\end{theorem} 
\begin{proof}
(i)$\Rightarrow$(ii). P.~G\"unther and F.~Pr\"ufer \cite{Gunther_Prufer} proved that every D'Atri space is ball-homogeneous. Local central geodesic reflections  preserve volume by the definition of a D'Atri space. Reflecting a small half-ball in its center takes the half-ball to the complementary half-ball. As these half-balls have equal volumes, their volume is half the volume of a ball of radius $r$.

(ii)$\Rightarrow$(i). If the volume of a half-ball of radius $r$ is $b(r)$, then for any unit tangent vector $\mathbf u\in SM$ and any small $r>0$, we have 
\[b(r)=\int_{B^+(r\mathbf u)}\theta(\mathbf v)\id \mathbf v=\int_0^r\rho^{n-1}\int_{S^+(\mathbf u)}\theta(\rho\mathbf v)\id \mathbf v\id \rho.\]
Differentiating with respect to $r$, we obtain
\[b'(r)=r^{n-1}\int_{S^+(\mathbf u)}\theta(r\mathbf v)\id \mathbf v.\]
This means that for any $p\in M$, and any small $r$, the hemispherical transformation of the function $\mathbf v\mapsto \theta(r\mathbf v)$ defined on the unit sphere $S_p$ is constant. This implies that the restriction of $\theta$ onto the sphere  $S_p(r)$ is an even function \cite[Proposition 3.4.11]{Groemer}, so local geodesic symmetries are volume-preserving. 

(i)$\Rightarrow$(iii) D'Atri spaces are real analytic by a result of Z.~I.~Szab\'o \cite{Szabo2}. As $\theta$ is even for a D'Atri space, $a_{2k+1}\equiv 0$ for all $k$, hence it is enough to deal with the even coefficients $a_{2k}$. For these, we show by induction on $k$ that if $\gamma$ is a unit speed geodesic, then $a_{2k}\circ \gamma'$ is constant. As $a_0\equiv 1$, the base case is obvious. Assume that the functions $a_{2j}\circ\gamma'$ are all constant for $0\leq j<k$. Then for any $t_0$ in the domain of $\gamma$, applying equation \eqref{eq:Vanhecke} for $\mathbf u=\gamma'(t_0)$, we obtain $0=\alpha_{2k}'(0)=(a_{2k}\circ \gamma')'(t_0)$, therefore $a_{2k}\circ \gamma'$ is constant.

(iii)$\Rightarrow$(iv). If $\gamma$ is an arbitrary  geodesic with $\gamma'(t_0)=r\mathbf u$, where $\mathbf u\in SM$, then 
\[
\theta(\gamma'(t))=\sum_{k=0}^{\infty}a_k(\gamma_{\mathbf u}'(r(t-t_0)))r^k.
\]
As the functions $a_k\circ\gamma_{\mathbf u}'$ are constant, $\theta\circ \gamma'$ is constant as well.

(iv)$\Rightarrow$(i). If (iv) holds, then for any $\mathbf v\in \mathring TM$ for which $-\mathbf v\in \mathring TM$ as well, we have
\[
\theta(\mathbf v)=\theta(\gamma_{\mathbf v}'(0))=\theta(\gamma_{\mathbf v}'(-1))=\theta(\iota(\gamma_{\mathbf v}'(-1)))=\theta(-\mathbf v),
\]
so local geodesic symmetries are volume-preserving.
\end{proof}
\begin{rem}\label{rem:proof_reduction}
	As it was noted in Section \ref{sec:2}, our proof of Theorem \ref{thm:main} reduces to a short proof of the fact that the tube property implies harmonicity. Indeed, by Theorems \ref{DAtri} and \ref{thm:D'Atri}, the tube property implies that the manifold is D'Atri, and half-ball homogeneous. For this reason,  assuming the tube property, one can skip all the lemmata of Section \ref{sec:3} and start the proof of harmonicity right after Lemma \ref{half-ball}. Then equations \eqref{eq:felgombvar}-\eqref{eq:bullet2} will be true without the error terms. In particular, equation \eqref{eq:bullet2} gives that the cosine transform of the even function $S_p\to \mathbb R$ defined by $\mathbf v\mapsto \theta(r\mathbf v)$ is the constant $2v(r)/r^{n-1}$ function for any $p\in M$, and $r<\inj(p)$. Using the properties of the cosine transform, this yields that $\theta(r\mathbf v)=v(r)/(\omega_{n-1}r^{n-1})$ depends only on $r$ for $\mathbf v\in SM$. 
\end{rem}

\section{Acknowledgements} The authors are grateful to professor Gudlaugur Thorbergsson for helpful discussions at the University of Cologne in 2013, that motivated the research presented above. They are also indebted to professor Eduardo Garc\'ia R\'io who proposed them the study of the total scalar curvature of tubes in harmonic manifolds during a personal communication at the Conference on Differential Geometry and its Applications in Brno in 2016. 
\bibliographystyle{ieeetr}
\bibliography{Tube}
\end{document}